\documentclass{amsart}
\usepackage[utf8]{inputenc}
\usepackage{amsmath}
\usepackage{amsfonts}
\usepackage{amssymb}
\usepackage{tikz-cd}
\usepackage{color,soul}
\usepackage{array}
\usepackage{zref-savepos}
\usepackage{amsthm}
\usepackage{multirow, bigstrut}
\usepackage[font={small,it}]{caption}
\usepackage{tikz}
\usepackage{bm}

\usepackage{todonotes} 

\usetikzlibrary{arrows}

\usepackage{hyperref}
\hypersetup{
    colorlinks=true,
    linkcolor=blue,
    filecolor=magenta,  
    urlcolor=cyan,
}

\usetikzlibrary{calc}
\usepackage{comment}
\usepackage{enumerate}
\usepackage{tikz}
\usetikzlibrary{shapes.geometric}

\newtheorem{theorem}{Theorem}[section]
\newtheorem*{theorem*}{Theorem}
\newtheorem{lemma}[theorem]{Lemma}
\newtheorem{proposition}[theorem]{Proposition}
\newtheorem{corollary}[theorem]{Corollary}

\theoremstyle{definition}

\newtheorem{remark}[theorem]{Remark}
\newtheorem{example}[theorem]{Example}

\theoremstyle{plain}

\newcommand{\C}{\mathbb{C}}
\newcommand{\R}{\mathbb{R}}
\newcommand{\N}{\mathbb{N}}

\DeclareMathOperator{\SL}{SL}
\DeclareMathOperator{\SO}{SO}

\title{Compatibility of Real-Rooted Polynomials with Mixed Signs}
\author{Jonathan Leake and Nick Ryder}

\begin{document}

\begin{abstract}
    We characterize compatible families of real-rooted polynomials, allowing both positive and negative leading coefficients.
    Our characterization naturally generalizes the same-sign characterization used by Chudnovsky and Seymour in their famous 2007 paper proving the real-rootedness of independence polynomials of claw-free graphs, thus fully settling a question left open in their paper.
    Our methods are generally speaking elementary, utilizing mainly linear algebra and the established theory of interlacing polynomials, along with a bit of invariant theory.
    %
\end{abstract}

\maketitle

\section{Introduction}

Univariate polynomials with only real roots, or \emph{real-rooted} polynomials, play an important role in the study of combinatorics and combinatorial inequalities. The main basic idea is that the coefficients of any real-rooted polynomial form a log-concave sequence, and thus one can prove log-concavity and unimodality statements by proving real-rootedness of associated polynomials. This technique is well-known and well-studied, dating back at least to
\cite{harper1967stirling,HL,comtet1974advanced} where it was applied to Stirling numbers, graph matchings, and Eulerian numbers respectively.
And more recently, a generalization of this technique to Lorentzian polynomials was used to resolve long-standing log-concavity conjectures on matroids \cite{log-concave3,BH}.
In general such techniques have been used to prove many log-concavity statements and other inequalities in combinatorics; see the surveys \cite{branden2015unimodality,brenti1994combinatorics,stanley1989log} and the references therein.

One general method for proving real-rootedness is the method of \emph{compatible families}, where a family of polynomials is called \emph{compatible} if all non-negative linear combinations have only real roots. This method can be particularly useful when working with graphs: since graph polynomials can often be written as sums of smaller graph polynomials, real-rootedness then follows inductively from compatibility of these smaller graph polynomials. This method has had various applications, even beyond log-concavity; in particular in proving existence of bipartite Ramanujan graphs of all degrees \cite{ramanujan}, and in the resolution of the Kadison-Singer conjecture \cite{marcus2015interlacingii}. See the survey \cite{liu2017method} for further discussion of these applications.


In \cite{CS}, Chudnovsky and Seymour used the method of compatible families to prove that the independence polynomial of a claw-free graph has only real roots (thus implying log-concavity and unimodality statements on its independent sets). Their proof utilized a nice characterization of compatible families of polynomials via \emph{interlacing polynomials}: roughly speaking, a polynomial $p$ interlaces $q$ if their roots alternate along the real line.
We state their characterization as follows.

\begin{theorem}[\cite{CS}; see also \cite{fell1980zeros,bartlett1988root,dedieu1992obreschkoff}] \label{thm:pos-leading-coeff-compat}
    Let $\bm{f} = (f_1,\ldots,f_n) \subset \R[t]$ be such that $f_i$ is real-rooted and has positive leading coefficient for all $f_i \in \bm{f}$. Then the following are equivalent:
    \begin{enumerate}
        \item $\bm{f}$ is compatible,
        \item $(f_i,f_j)$ is compatible for all $i,j$,
        \item there is a real-rooted $g \in \R[t]$ which interlaces every $f_i \in \bm{f}$.
    \end{enumerate}
\end{theorem}

At the end of \cite{CS}, Chudnovsky and Seymour then ask for an analogue of Theorem~\ref{thm:pos-leading-coeff-compat} without the assumption that the leading coefficients of the polynomials are all positive. They provide a counterexample to the naive analogue of $(1) \iff (2)$, but leave open the question of an appropriate version of $(1) \iff (3)$. Partial results towards this question were obtained in \cite{liu2012polynomials}, for the $(3) \implies (1)$ direction and also for the $n=2$ case. The main goal of this paper is then to fully answer the question of Chudnovsky and Seymour.

\subsection{Previous work}

Just after posting this paper, Petter Br\"and\'en informed the authors of previous work which already essentially implies our main result. Specifically Theorem~\ref{thm:main}: $(1) \iff (2)$ is a corollary of Lemma 4 of \cite{bialas2012few}.

In fact, within the control theory literature there are a number of results regarding root locations of convex combinations of polynomials, with results often holding for polynomials with roots in a wide range of regions within the complex plane (including the real line). This includes the so-called ``edge theorem'' of \cite{bartlett1988root} (see also Theorem 1 of \cite{sideris1989edge}) which already implies Theorem \ref{thm:pos-leading-coeff-compat}: $(1) \iff (2)$. It appears that this line of work is relatively unknown to the geometry of polynomials community; we are unaware of references to these and other related papers in the interlacing polynomials literature.

\subsection{Our results}

To state our results, we need a bit more notation. Given real-rooted $f \in \R[t]$ we denote its roots (counted with multiplicities) by $\lambda_1(f) \geq \lambda_2(f) \geq \cdots$. Given real-rooted $f,g \in \R[t]$ with positive leading coefficients we say that $g$ \emph{interlaces}\footnote{Note that some authors reverse the order of $f$ and $g$ in this notation.} $f$, denoted $f \ll g$, if $\deg(f) \in \{\deg(g), \deg(g)-1\}$ and
\[
    \lambda_1(g) \geq \lambda_1(f) \geq \lambda_2(g) \geq \lambda_2(f) \geq \lambda_3(g) \geq \cdots.
\]
We also allow negative leading coefficients by adding the rule that $f \ll g$ implies $g \ll -f$ for all real-rooted $f,g \in \R[t]$. We also use the convention that $0 \ll f$ and $f \ll 0$ for any real-rooted $f \in \R[t]$. Condition $(3)$ of Theorem~\ref{thm:pos-leading-coeff-compat} is then equivalently stated as: there is a real-rooted $g \in \R[t]$ such that $f_i \ll g$ for all $i$. And finally, recall from above that a family of polynomials $\bm{f} \subset \R[t]$ is \emph{compatible} if all non-negative linear combinations of $\bm{f}$ are real-rooted. With this, we can state our main result.

\begin{theorem}[see also \cite{bialas2012few}] \label{thm:main}
    Let $\bm{f} = (f_1,\ldots,f_n) \subset \R[t]$ be such that $f_i$ is real-rooted for all $f_i \in \bm{f}$. Then the following are equivalent:
    \begin{enumerate}
        \item $\bm{f}$ is compatible,
        \item $(f_i,f_j,f_k)$ is compatible for all $i,j,k$.
    \end{enumerate}
    If we further assume that no convex combination of $\bm{f}$ is the zero polynomial, then the following is also equivalent to the above conditions:
    \begin{enumerate}
        \setcounter{enumi}{2}
        \item there is a real-rooted $g \in \R[t]$ such that $f_i \ll g$ for all $i$.
    \end{enumerate}
\end{theorem}

Let us note that condition $(3)$ of Theorem~\ref{thm:main} always implies $(1)$, even without the extra assumption that no convex combination of $\bm{f}$ is the zero polynomial. (This is proven in \cite[Theorem~2.3 and Remark~2.4]{liu2012polynomials}; we give a short proof of this fact in Section~\ref{sec:3_implies_1} below.) On the other hand, this extra assumption cannot be dropped in general (see Example~\ref{ex:counter}). Finally, in the case that $n=2$ this extra assumption implies $af_1 = -bf_2$, and thus $(1)$ and $(3)$ are equivalent in general in this case.

Given Theorem~\ref{thm:main}, the natural next result is a characterization of compatible families $(f_1,f_2,f_3) \subset \R[t]$ of size $3$. Combined with Theorem~\ref{thm:main} $(1) \iff (2)$, this yields a characterization of compatible families of polynomials based entirely upon interlacing properties of the polynomials.

\begin{proposition} \label{prop:main-compatible-triples}
    Let $(f_1,f_2,f_3) \subset \R[t]$ be such that $f_i$ is real-rooted for all $i \in \{1,2,3\}$. Then $(f_1,f_2,f_3)$ is compatible if and only if one of the following conditions holds:
    \begin{enumerate}
        \item there is a real-rooted $g \in \R[t]$ such that $f_1,f_2,f_3 \ll g$,
        \item there is a choice of $\{i,j,k\} = \{1,2,3\}$, $a \geq 0$, and $b,c > 0$ such that $af_i + bf_j + cf_k \equiv 0$ and either $f_i \ll f_j$ or $f_j \ll f_i$.
    \end{enumerate}
\end{proposition}

Finally, we are also able to say something more topological about the simplicies consisting of all convex combinations of $\bm{f}$. In what follows, let $\Delta(\bm{f})$ denote the simplex generated by $\bm{f}$, let $\R^d[t]$ denote the set of real polynomials of degree at most $d$, and let $\|\cdot\|$ denote any norm on $\R^d[t]$.

\begin{theorem} \label{thm:simplex-main}
    Let $\bm{f} = (f_1,\ldots,f_n) \subset \R^d[t]$ be compatible and such that $0 \not\in \Delta(\bm{f})$. Then for all $\epsilon > 0$ small enough there exists $\tilde{\bm{f}} = (\tilde{f}_1,\ldots,\tilde{f}_n) \subset \R^d[t]$ such that $\|f_i - \tilde{f}_i\| < \epsilon$ for all $i$ and 
    $\Delta(\tilde{\bm{f}})$ is contained in the interior of the set of real-rooted polynomials in $\R^d[t]$.
\end{theorem}

Theorem~\ref{thm:simplex-main} does not consider the case when $0 \in \Delta(\bm{f})$. While we have not fully explored this case, we know that the conclusion of Theorem~\ref{thm:simplex-main} does not hold in full generality (see the discussion following Example~\ref{ex:counter}).

\subsection{Proof ideas, and counterexamples}

The results of this paper essentially follow from previously known theory of interlacing and compatible polynomials, along with some linear algebra and a bit of invariant theory of $\SL_2(\R)$. That said, here we will try to give a high-level sense of the main ideas of our proofs, and why certain conditions of our results cannot be dropped.

The bulk of the proofs is devoted to proving that condition $(2)$ of Theorem~\ref{thm:main} implies conditions $(1)$ and $(3)$, and so let's suppose we have a family $\bm{f}$ for which condition $(2)$ holds. One of the key insights of the proof then starts with a well-known idea: if $p_s := (1-s) \cdot p_0 + s \cdot p_1$ for $s \in [0,1]$ is such that $p_0$ is real-rooted and $p_1$ is not, then there must be some $s \in [0,1]$ for which $p_s$ is real-rooted with a non-simple root. In what follows, we are able to show that the conditions under which some convex combination $g$ of $\bm{f}$ has a non-simple root $r$ are very restrictive: $r$ must be a root of $f_i$ for some $i$ (Proposition~\ref{prop:3-compat-exists-root}), and the signs of $f_i(r)$ and $g''(r)$ are determined for all $i$ (Proposition~\ref{prop:double-root-signs}). These conditions allow us to perturb $\bm{f}$ to force all zeros of all convex combinations of $\bm{f}$ to be simple, thus implying compatibility.

That said, most of the arguments described above implicitly require that no convex combination of $\bm{f}$ is the zero polynomial. This is why condition $(3)$ of Theorem~\ref{thm:main} requires an extra assumption, compared to Theorem~\ref{thm:pos-leading-coeff-compat}. And further, this condition cannot be removed, which we now demonstrate with the following illustrative example.

\begin{example} \label{ex:counter}
    For any $r \geq 0$, define
    \[
        f(t) = r^2-t^2, \quad g(t) = t^2+2t-3, \quad h(t) = t^2-2t-3,
    \]
    with roots $(-r,r)$, $(-3,1)$, and $(-1,3)$ respectively. Thus all pairs of $\{f,g,h\}$ are compatible if and only if $r \in [1,3]$, and $\{f,g,h\}$ is compatible if and only if $r = \sqrt{3}$. Note that $r = \sqrt{3}$ is precisely the point where some convex combination of $\{f,g,h\}$ is the zero polynomial.
\end{example}

We now make a few observations regarding this example:
\begin{itemize}
    \item For $r \in [1,3] \setminus \{\sqrt{3}\}$, Example~\ref{ex:counter} gives a counterexample (similar to that of \cite{CS}) to the claim that pairwise compatibility of $\bm{f}$ is sufficient to imply compatibility of $\bm{f}$.

    \item For $r = \sqrt{3}$, Example~\ref{ex:counter} gives a counterexample to the claim that one can drop the extra assumption required for the equivalence of condition $(3)$ in Theorem~\ref{thm:main}. This is because the ordering of the roots of $f,g,h$ in this case makes it impossible to find some real-rooted $p$ for which $f,g,h \ll p$.

    \item For $r = \sqrt{3}$, continuity of roots implies no perturbation of $\Delta(\{f,g,h\})$ can lie in the interior of the set of real-rooted polynomials (see Theorem~\ref{thm:simplex-main}). Otherwise, we would obtain condition $(3)$ of Theorem~\ref{thm:main} for this perturbation, which would be a contradiction.
\end{itemize}

This suggests Example~\ref{ex:counter} (and its implications) as a reason why the general case is harder to handle than the case where all leading coefficients are of the same sign.

\subsection{Roadmap}

In Section~\ref{sec:prelims} we give notation, state various preliminary results on interlacing and the action of $\SL_2(\R)$, and prove the $(3) \implies (1)$ case of Theorem~\ref{thm:main}.
In Section~\ref{sec:basic-k-compat} we prove some basic and useful results for $k$-compatible families.
In Section~\ref{sec:proof-main} we prove the $(1) \iff (2)$ case of Theorem~\ref{thm:main}.
In Section~\ref{sec:main-2-implies-3} we prove the $(2) \implies (3)$ case of Theorem~\ref{thm:main}.
In Section~\ref{sec:simplex} we prove Theorem~\ref{thm:simplex-main}.
In Section~\ref{sec:proof-main-triples} we prove Proposition~\ref{prop:main-compatible-triples}.

\section{Preliminaries} \label{sec:prelims}

\subsection{Notation}

Let $\R^d[t]$ denote the vector space of real univariate polynomials of degree at most $d$. Given a finite list of polynomials $\bm{f} = (f_1,\ldots,f_n) \subset \R^d[t]$, we say that $\bm{f}$ is \textbf{proper} if there is no convex combination which is the zero polynomial, we say that $\bm{f}$ is \textbf{compatible} if every convex combination is real-rooted, and we say that $\bm{f}$ is \textbf{$k$-compatible} if every convex combination of at most $k$ polynomials in $\bm{f}$ is real-rooted. Further, we say that $f \in \R^d[t]$ has a root \textbf{at infinity (of multiplicity $m$)} if the degree of $f$ is $d-m$. The zero polynomial has a root of multiplicity $+\infty$ at every point in $\R \cup \{\infty\}$.

As above, given real-rooted $f \in \R[t]$ we denote its roots (counted with multiplicities) by $\lambda_1(f) \geq \lambda_2(f) \geq \cdots$. Given real-rooted $f,g \in \R[t]$ with positive leading coefficients we say that $g$ \textbf{interlaces} $f$, denoted $f \ll g$, if $\deg(f) \in \{\deg(g), \deg(g)-1\}$ and
\[
    \lambda_1(g) \geq \lambda_1(f) \geq \lambda_2(g) \geq \lambda_2(f) \geq \lambda_3(g) \geq \cdots.
\]
We also allow negative leading coefficients by adding the rule that $f \ll g$ implies $g \ll -f$ for all real-rooted $f,g \in \R[t]$.\footnote{\cite{liu2012polynomials} uses the separate notion of ``modification common interleaver'' for the case of leading coefficients of different signs.} We also use the convention that $0 \ll f$ and $f \ll 0$ for any real-rooted $f \in \R[t]$, and that $0 \ll 0$. Finally, the \textbf{Wronskian} of $f,g \in \R[t]$, denoted $W[f,g]$, is defined as $W[f,g] = f' \cdot g - g' \cdot f$.

\subsection{Interlacing polynomials}

Here we discuss some previously known results on interlacing polynomials. We will follow the presentation of Wagner. In \cite{wagner2011multivariate}, Wagner defines $\ll$ slightly differently, but we will show that our definitions coincide. By \cite[Section~2.3]{wagner2011multivariate}, we first have the following.

\begin{lemma} \label{lem:wronskian-pre}
    Fix real-rooted polynomials $f,g \in \R[t]$ and suppose either $f \ll g$ or $g \ll f$. Then either $W[f,g] \leq 0$ on all of $\R$ or $W[f,g] \geq 0$ on all of $\R$.
\end{lemma}

\begin{lemma} \label{lem:wronskian}
    Fix real-rooted polynomials $f,g \in \R[t]$ and suppose either $f \ll g$ or $g \ll f$. Then $f \ll g$ if and only if $W[f,g] \leq 0$ on all of $\R$.
\end{lemma}
\begin{proof}
    If either $f \equiv 0$ or $g \equiv 0$, then the result is immediate. Otherwise note that $W[f,g] = W[g,-f]$, and thus by definition of $\ll$ we may assume that $f,g$ both have positive leading coefficients. Now we compute
    \[
        R'(t) = \left(\frac{f}{g}\right)'(t) = \frac{f' \cdot g - g' \cdot f}{g^2} = \frac{W[f,g]}{g^2}.
    \]
    Thus by Lemma~\ref{lem:wronskian-pre}, $R'(t)$ has the same sign on all of $\R$. The above computation also shows that without loss of generality, we may assume that $f,g$ share no roots by possibly dividing by $(t-r)$ for every shared root $r$.

    With this, let $t_0 = \lambda_1(g) + \epsilon$ and $t_1 = \lambda_1(g) + \frac{1}{\epsilon}$ for small $\epsilon > 0$. Since $g$ has no zeros in $[t_0,t_1]$, the function $R'(t)$ is continuous on $[t_0,t_1]$. If $f \ll g$ then $R(t_0)$ is large and positive, and if $g \ll f$ then $R(t_0)$ is large and negative. Further, $R(t_1)$ approaches the ratio of the leading coefficients of $f,g$. Thus if $f \ll g$ then $R(t_1)$ approaches a positive constant or $0$ as $\epsilon \to 0$, and if $g \ll f$ then $R(t_1)$ approaches a positive constant or $+\infty$ as $\epsilon \to 0$. Therefore if $f \ll g$ then $R'(t) \leq 0$, and if $g \ll f$ then $R'(t) \geq 0$. Thus the same holds for $W[f,g]$.
\end{proof}

The next result is one of the most important in the theory; see \cite[Section~2.3]{wagner2011multivariate} for discussion and further references.

\begin{theorem}[Hermite-Kakeya-Obreschkoff (HKO) Theorem] \label{thm:hko}
    Let $f,g \in \R[t]$ be real-rooted polynomials. Then $af + bg$ is real-rooted for all $a,b \in \R$ if and only if either $f \ll g$ or $g \ll f$.
\end{theorem}

\subsection{The action of $\SL_2(\R)$}

The vector space $\R^d[t]$ can be equipped with a group action of $\SL_2(\R)$ in the usual way. Given $p \in \R^d[t]$ and $\phi = \left[\begin{smallmatrix} \alpha & \beta \\ \gamma & \delta \end{smallmatrix}\right] \in \SL_2(\R)$, we define
\[
    \phi^{-1} \cdot p(t) = (\gamma t+\delta)^d \cdot p\left(\frac{\alpha t+\beta}{\gamma t+\delta}\right).
\]
Given $r \in \R \cup \{\infty\}$ we also define $\phi \cdot r = \frac{\alpha r+\beta}{\gamma r+\delta}$, which implies $p(r) = 0$ if and only if $\phi \cdot p(\phi \cdot r) = 0$. (Here, $\phi \cdot \infty = \frac{\alpha}{\gamma}$.) This action clearly preserves the set of real-rooted polynomials in $\R^d[t]$. See \cite{wagner2001logarithmic} for further discussion of the connections of this group action to log-concavity.

We will mainly utilize this $\SL_2(\R)$ action to remove roots at infinity and to adjust polynomials to have certain nice properties. Thus we need the following basic result. While we could not find this result in the literature, it was probably already essentially known to the experts.

\begin{lemma} \label{lem:sl2-interlacing}
    Given $\phi \in \SL_2(\R)$ and real-rooted $f,g \in \R^d[t]$, we have that $f \ll g$ if and only if $\phi \cdot f \ll \phi \cdot g$.
\end{lemma}
\begin{proof}
    Since $\SL_2(\R)$ is a group, we only need to show that $f \ll g$ implies $\phi \cdot f \ll \phi \cdot g$. Since $\phi$ acts as a linear map on $\R^d[t]$, if all linear combinations of $f$ and $g$ are real-rooted, then all linear combinations of $\phi \cdot f$ and $\phi \cdot g$ are real-rooted. Thus by HKO (Theorem~\ref{thm:hko}), $f \ll g$ implies either $\phi \cdot f \ll \phi \cdot g$ or $\phi \cdot g \ll \phi \cdot f$. Finally, a straightforward computation shows that the Wronskian, considered as a map $W: \R^d[t] \times \R^d[t] \to \R^{2d-2}[t]$ is invariant under the action of $\SL_2(\R)$; that is, $W[\phi \cdot f, \phi \cdot g] = \phi \cdot W[f,g]$. Since $2d-2$ is even, Lemma~\ref{lem:wronskian} therefore implies $\phi \cdot f \ll \phi \cdot g$.
\end{proof}

\begin{remark} \label{rem:rotations}
    Let $f \in \R^d[t]$ be real-rooted, and consider the roots of $f$ to be on the unit circle via stereographic projection to $\R \cup \{\infty\}$. Then $\SO_2(\R) \subset \SL_2(\R)$ acts on $f$ by rotating its roots about the circle. Further, $\SO_2(\R)$ is a compact group and thus preserves some norm on $\R^d[t]$ (e.g., $\|p\| := \int_{\SO_2(\R)} \|\phi \cdot p\|' \, d\phi$ for any norm $\|\cdot\|'$). We will frequently use $\SO_2(\R)$, along with $\SL_2(\R)$, to perturb polynomials to have desirable properties.
\end{remark}



\subsection{Proof of Theorem~\ref{thm:main}: $(3) \implies (1)$} \label{sec:3_implies_1}

We now give a quick proof of the $(3) \implies (1)$ case of Theorem~\ref{thm:main}. As stated in the introduction regarding this case of Theorem~\ref{thm:main}, we do not require the extra assumption of properness on $\bm{f}$, and this is already proven in \cite[Theorem~2.3 and Remark~2.4]{liu2012polynomials}. We prove this here mainly to demonstrate the utility of the action of $\SL_2(\R)$ in the theory of interlacing polynomials.

\begin{proposition}
    Let $\bm{f} = (f_1,\ldots,f_n) \subset \R^d[t]$ be such that $f_i$ is real-rooted for all $f_i \in \bm{f}$. If there is a real-rooted $g \in \R^d[t]$ such that $f_i \ll g$ for all $i$, then $\bm{f}$ is compatible.
\end{proposition}
\begin{proof}
    By possibly applying an $\SL_2(\R)$ perturbation to $\bm{f}$ and $g$, we may assume that  $g$ is of degree exactly $d$ and $f_i$ is of degree exactly $d$ for all $i$. By possibly negating all polynomials, we may assume that $g$ has positive leading coefficient. And finally, by possibly permuting $\bm{f}$, we may assume that $f_1,\ldots,f_m$ have positive leading coefficients and $f_{m+1},\ldots,f_n$ have negative leading coefficients for some $0 \leq m \leq n$.
    
    Thus we have $f_1,\ldots,f_m \ll g \ll -f_{m+1}, \ldots, -f_n$, which implies
    \[
        \lambda_1(f_1), \ldots, \lambda_1(f_m), \lambda_2(f_{m+1}), \ldots, \lambda_2(f_n) \leq \lambda_1(g) \leq \lambda_1(f_{m+1}), \ldots, \lambda_1(f_n).
    \]
    Thus if we increase $\lambda_1(f_k)$ by $\epsilon > 0$ for all $m+1 \leq k \leq n$, we still have $f_i \ll g$ for all $i$, and we also have that $t_0 = \lambda_1(g) + \frac{\epsilon}{2}$ is such that
    \[
        \lambda_1(f_1), \ldots, \lambda_1(f_m), \lambda_2(f_{m+1}), \ldots, \lambda_2(f_n) \leq \lambda_1(g) < t_0 < \lambda_1(f_{m+1}), \ldots, \lambda_1(f_n).
    \]
    Letting $\phi \in \SL_2(\R)$ be a rotation which maps $t_0$ to $\infty$ (see Remark~\ref{rem:rotations}), we have that $\phi \cdot g$ has positive leading coefficient, $\phi \cdot f_i$ has positive leading coefficient for all $i$, and $\phi \cdot f_i \ll \phi \cdot g$ for all $i$ by Lemma~\ref{lem:sl2-interlacing}. Thus $\phi \cdot \bm{f}$ is compatible by Theorem~\ref{thm:pos-leading-coeff-compat}, and therefore so is $\bm{f}$. Limiting $\epsilon \to 0$ then implies the desired result.
\end{proof}

\section{Results for $k$-compatible families} \label{sec:basic-k-compat}

In this section we provide some basic results for $k$-compatible families of polynomials, with a specific focus on $3$-compatible families. The goal here is to demonstrate that the conditions under which a convex combination of a $3$-compatible family $\bm{f}$ has a non-simple root are rather restrictive. Specifically, Proposition~\ref{prop:3-compat-exists-root} shows that non-simple roots can only appear at finitely many points, and Proposition~\ref{prop:double-root-signs} shows that a non-simple root at $r$ determines the sign of $f(r)$ for all $f \in \bm{f}$. We will then use these basic results to prove some of our main results in Section~\ref{sec:proof-main}.

\subsection{Basic results via linear algebra}

The main goal of this section is to prove Corollary~\ref{cor:k-compat-2-polys-root}, which allows us to restrict to only checking pairs of polynomials of a $k$-compatible family (for $k \geq 3$) when looking for non-simple roots.

\begin{lemma} \label{lem:root-linear-reduction}
    Fix $\bm{f} = (f_1,\ldots,f_n) \subset \R^d[t]$ and suppose some convex combination of $\bm{f}$ has a root at $r \in \R \cup \{\infty\}$ of multiplicity at least $m$. Then some convex combination of at most $m+1$ polynomials in $\bm{f}$ has a root at $r$ of multiplicity at least $m$.
\end{lemma}
\begin{proof}
    By possibly applying an $\SL_2(\R)$ perturbation to $\bm{f}$, we may assume that $r \neq \infty$, and by possibly considering a sublist of $\bm{f}$, we may assume that some strictly positive convex combination of $\bm{f}$ has a root at $r$ of multiplicity at least $m$. Consider the $m \times n$ matrix
    \[
        A = \begin{bmatrix}
            f_1(r) & f_2(r) & \cdots & f_n(r) \\
            f_1'(r) & f_2'(r) & \cdots & f_n'(r) \\
            \vdots & \vdots & \ddots & \vdots \\
            f_1^{(m-1)}(r) & f_2^{(m-1)}(r) & \cdots & f_n^{(m-1)}(r)
        \end{bmatrix}.
    \]
    The kernel of the matrix $A$ then corresponds precisely to the set of all linear combinations of $\bm{f}$ with a root at $r$ of multiplicity at least $m$, and this kernel is of dimension at least $n-m$. By assumption this kernel intersects the strict positive orthant $\R_{>0}^n$, and by a dimension count it also intersects some non-zero face of $\R_{>0}^n$ of dimension at most $n-(n-m-1)$. Thus there exists some non-zero vector in $\R_{\geq 0}^n$ with at most $m+1$ non-zero entries which is in the kernel of $A$, and this implies the desired result.
\end{proof}

\begin{lemma} \label{lem:compatible-common-roots}
    Let $\bm{f} = (f_1,\ldots,f_n) \subset \R^d[t]$ be proper and compatible. If some strictly positive convex combination of $\bm{f}$ has a root at $r \in \R \cup \{\infty\}$ of multiplicity at least $m$, then every $f \in \bm{f}$ has a root at $r$ of multiplicity at least $m-1$.
\end{lemma}
\begin{proof}
    By possibly applying an $\SL_2(\R)$ perturbation to $\bm{f}$, we may assume that $r \neq \infty$. If $m \leq 1$ then the result is trivial, so we assume $d \geq m \geq 2$. Let $g \not\equiv 0$ denote the strictly positive convex combination of $\bm{f}$ which has a root at $r$ of multiplicity at least $m$. By assumption $g \pm \epsilon f$ is real-rooted for all $f \in \bm{f}$ and all $\epsilon > 0$ small enough. This implies $f(r) = 0$ for all $f \in \bm{f}$. Now divide all $f \in \bm{f}$ by $(t-r)$ and reduce $d$ and $m$ by 1, and the result follows by induction on $m$.
\end{proof}

\begin{proposition} \label{prop:compat-2-polys-root}
    Let $\bm{f} = (f_1,\ldots,f_n) \subset \R^d[t]$ be proper and compatible. If some convex combination of $\bm{f}$ has a root at $r \in \R \cup \{\infty\}$ of multiplicity at least $m$, then some convex combination of at most $2$ polynomials in $\bm{f}$ has a root at $r$ of multiplicity at least $m$.
\end{proposition}
\begin{proof}
    By possibly applying an $\SL_2(\R)$ perturbation to $\bm{f}$, we may assume that $r \neq \infty$, and by possibly considering a sublist of $\bm{f}$, we may assume that some strictly positive convex combination of $\bm{f}$ has a root at $r$ of multiplicity at least $m$. By Lemma~\ref{lem:compatible-common-roots}, every $f \in \bm{f}$ has a root at $r$ of multiplicity at least $m-1$. Now divide all $f_i \in \bm{f}$ by $(t-r)^{m-1}$ to obtain $g_i$, and define $\bm{g} = (g_1,\ldots,g_n) \subset \R^{d-m+1}[t]$ so that some strictly positive convex combination of $\bm{g}$ has a root at $r$ (of multiplicity at least $1$). Lemma~\ref{lem:root-linear-reduction} then implies some convex combination of at most $2$ polynomials in $\bm{g}$ has a root at $r$ of multiplicity at least $1$. This implies some non-negative linear combination of at most $2$ polynomials in $\bm{f}$ has a root at $r$ of multiplicity at least $m$. 
\end{proof}

\begin{corollary} \label{cor:k-compat-2-polys-root}
    Let $\bm{f} = (f_1,\ldots,f_n) \subset \R^d[t]$ be proper and $k$-compatible, and fix $m \leq k-1$. If some convex combination of $\bm{f}$ has a root at $r \in \R \cup \{\infty\}$ of multiplicity at least $m$, then some convex combination of at most $2$ polynomials in $\bm{f}$ has a root at $r$ of multiplicity at least $m$.
\end{corollary}
\begin{proof}
    By Lemma~\ref{lem:root-linear-reduction}, some convex combination of at most $m+1 \leq k$ polynomials in $\bm{f}$ has a root at $r$ of multiplicity at least $m$. The desired result then follows from $k$-compatibility and Proposition~\ref{prop:compat-2-polys-root}.
\end{proof}

\subsection{Non-simple roots of $3$-compatible families}

The goal of this section is to prove (in in Proposition~\ref{prop:3-compat-exists-root} and Proposition~\ref{prop:double-root-signs}) the restrictions on non-simple roots of $3$-compatible families which were claimed at the start of Section~\ref{sec:basic-k-compat}.

\begin{corollary} \label{cor:3-compat-multiple-roots}
    Let $\bm{f} = (f_1,\ldots,f_n) \subset \R^d[t]$ be proper and $3$-compatible. If some convex combination of $\bm{f}$ has a non-simple root at $r \in \R \cup \{\infty\}$, then some convex combination of at most $2$ polynomials in $\bm{f}$ has a non-simple root at $r$.
\end{corollary}
\begin{proof}
    Immediately follows from Corollary~\ref{cor:k-compat-2-polys-root} with $k=3$ and $m=2$.
\end{proof}

\begin{proposition} \label{prop:3-compat-exists-root}
    Let $\bm{f} = (f_1,\ldots,f_n) \subset \R^d[t]$ be proper and $3$-compatible. If some convex combination of $\bm{f}$ has a non-simple root at $r \in \R \cup \{\infty\}$, then some $f \in \bm{f}$ has a root at $r$.
\end{proposition}
\begin{proof}
    By Corollary~\ref{cor:3-compat-multiple-roots},
    some convex combination of at most $2$ polynomials in $\bm{f}$ has a non-simple root at $r$. The desired result then follows by restricting to these two polynomials and applying Lemma~\ref{lem:compatible-common-roots} if necessary.
\end{proof}

\begin{lemma} \label{lem:3-compat-triple-root}
    Let $\bm{f} = (f_1,\ldots,f_n) \subset \R^d[t]$ be proper and $3$-compatible. If some convex combination of at most $3$ polynomials in $\bm{f}$ has a root at $r$ of multiplicity at least $3$, then every $f \in \bm{f}$ has a root at $r$.
\end{lemma}
\begin{proof}
    By possibly applying an $\SL_2(\R)$ perturbation to $\bm{f}$, we may assume that $r \neq \infty$. By Proposition~\ref{prop:compat-2-polys-root}, some convex combination of at most $2$ polynomials in $\bm{f}$ has a root at $r$ of multiplicity at least $3$. Without loss of generality, we may assume some convex combination $a_1 f_1 + a_2 f_2$ has a root at $r$ of multiplicity at least $3$. By $3$-compatibility $a_1 f_1 + a_2 f_2 + \epsilon f$ is real-rooted for all $f \in \bm{f}$ and $\epsilon > 0$, and this implies every $f \in \bm{f}$ has a root at $r$.
\end{proof}

\begin{proposition} \label{prop:double-root-signs}
    Let $\bm{f} = (f_1,\ldots,f_n) \subset \R^d[t]$ be proper and $3$-compatible, such that there is no $r \in \R \cup \{\infty\}$ for which $f(r) = 0$ for all $f \in \bm{f}$. If $g$ is a convex combination of at most $3$ polynomials in $\bm{f}$ such that $g$ has a non-simple root at $r \in \R$, then $g''(r) \neq 0$ and $g''(r) \cdot f(r) \leq 0$ for all $f \in \bm{f}$.
\end{proposition}
\begin{proof}
    Without loss of generality, let $g = a_1 f_1 + a_2 f_2 + a_3 f_3$ be a convex combination with a non-simple root at $r \in \R$. By Corollary~\ref{cor:3-compat-multiple-roots}, without loss of generality we may assume that $h = b_1 f_1 + b_2 f_2$ is also a convex combination with a non-simple root at $r$. By Lemma~\ref{lem:3-compat-triple-root} and our assumptions on $\bm{f}$, the multiplicity of $r$ in any convex combination of $g$ and $h$ equals $2$, which implies $g''(r),h''(r) \neq 0$ and $g''(r) \cdot h''(r) > 0$. Now by $3$-compatibility, $h + \epsilon f$ is real-rooted for all $f \in \bm{f}$ and $\epsilon > 0$. This implies $h''(r) \cdot f(r) \leq 0$ for all $f \in \bm{f}$. Since $h''(r) \cdot g''(r) > 0$, this implies the desired result.
\end{proof}

\section{Proof of Theorem~\ref{thm:main}: $(1) \iff (2)$} \label{sec:proof-main}

In this section we prove Theorem~\ref{thm:main-compat-perturb} which is the main technical result from which most of our results are derived. This theorem gives an explicit way to perturb a proper $3$-compatible family so that all convex combinations have simple real roots. We use it here to directly prove the $(1) \iff (2)$ case of Theorem~\ref{thm:main}.

\subsection{The simple roots case}

We first handle the much easier case in which the product of all polynomials of a given $3$-compatible family has simple roots. For this case, the basic results of Section~\ref{sec:basic-k-compat} essentially immediately imply the $(1) \iff (2)$ case of Theorem~\ref{thm:main}.

\begin{proposition} \label{prop:3-compat-real-rooted}
    Let $\bm{f} = (f_1,\ldots,f_n) \subset \R^d[t]$ be proper and $3$-compatible. If every convex combination of at most $2$ polynomials in $\bm{f}$ has only simple roots, then $\bm{f}$ is compatible and every convex combination of $\bm{f}$ has only simple real roots.
\end{proposition}
\begin{proof}
    By Corollary~\ref{cor:3-compat-multiple-roots}, the real roots of every convex combination of $\bm{f}$ are simple. Now suppose some convex combination of $\bm{f}$ has non-real roots, call it $g$. Since $f \in \bm{f}$ is real-rooted with only simple roots by assumption, real coefficients and continuity of roots imply some convex combination of $f$ and $g$ must have a non-simple real root. This is a contradiction, and thus every convex combination of $\bm{f}$ is real-rooted.
\end{proof}

\begin{corollary} \label{cor:simple-3-compat}
    Let $\bm{f} = (f_1,\ldots,f_n) \subset \R^d[t]$ be proper and $3$-compatible. If the roots of $\prod_{i=1}^n f_i \in \R^{nd}[t]$ are simple, then $\bm{f}$ is compatible and every convex combination of $\bm{f}$ has only simple real roots.
\end{corollary}
\begin{proof}
    So as to get a contradiction by Proposition~\ref{prop:3-compat-real-rooted}, suppose without loss of generality that some convex combination $af_1 + bf_2$ has a non-simple root at $r \in \R \cup \{\infty\}$. Thus either $f_1$ or $f_2$ has a non-simple root at $r$, or by Lemma~\ref{lem:compatible-common-roots} both $f_1$ and $f_2$ have a root at $r$. This implies $f_1 \cdot f_2$ has a non-simple root at $r$, a contradiction.
\end{proof}

\subsection{The non-simple roots case}

The case where some convex combinations of the polynomials of a $3$-compatible family $\bm{f}$ have non-simple roots is more difficult to handle. To do so, we prove Theorem~\ref{thm:main-compat-perturb} by utilizing the restrictions on non-simple roots proven in Section~\ref{sec:basic-k-compat}. The main idea is: because there are finitely many non-simple root locations and because the sign of all $f \in \bm{f}$ is determined at a non-simple root, we can make the non-simple roots simple via the perturbation $\bm{f} \mapsto \bm{f} + \epsilon \cdot \sum_{f \in \bm{f}} f$. This is done explicitly in Lemma~\ref{lem:two-real-roots} and in the proof of Theorem~\ref{thm:main-compat-perturb}. We finally prove the $(1) \iff (2)$ case of Theorem~\ref{thm:main} for proper $\bm{f}$ in Corollary~\ref{cor:main-proper-case}.

We first need the following result from \cite{cucker1989alternate} which gives a more refined description of the continuity of the roots (in $\C \cup \{\infty\}$) of polynomials in $\C^d[t]$. Let $\|\cdot\|$ denote any fixed norm on the vector space $\C^d[t]$.

\begin{theorem}[\cite{cucker1989alternate}, Theorem 3] \label{thm:continuous-roots}
    Let $f \in \C^d[t]$ be a non-zero complex polynomial with distinct finite roots $r_1,\ldots,r_k \in \C$ and respective multiplicities $m_1,\ldots,m_k$. Let $U_1,\ldots,U_k$ be disjoint discs centered at $r_1,\ldots,r_k$ with radii $\epsilon > 0$ and contained in the open disc centered at $0$ with radius $\epsilon^{-1}$. There exists $\delta > 0$ such that for $g \in \C^d[t]$, $\|g - f\| < \delta$ implies $g$ has $m_i$ roots in each $U_i$ and $\deg(g) - \deg(f)$ finite roots with modulus greater than $\epsilon^{-1}$, all counted with multiplicity. (Thus $g$ has $d - \deg(f)$ roots, including $\infty$, with modulus greater than $\epsilon^{-1}$.)
\end{theorem}

\begin{corollary} \label{cor:continuous-stability}
    Let $f \in \C^d[t]$ be a non-zero complex polynomial with no roots in a closed and bounded set $S \subset \C$. There exists $\delta > 0$ such that for $g \in \C^d[t]$, $\|g - f\| < \delta$ implies $g$ has no roots in $S$.
\end{corollary}
\begin{proof}
    Since $S$ is closed and bounded, we can choose $\epsilon$-discs about the roots of $f$ with $\epsilon > 0$ small enough so that the associated discs $U_i$ in Theorem~\ref{thm:continuous-roots} are disjoint from $S$ and so that $S$ is contained in the open disc centered at $0$ with radius $\epsilon^{-1}$. Applying Theorem~\ref{thm:continuous-roots} then gives the desired result.
\end{proof}

\begin{lemma} \label{lem:two-real-roots}
    Let $B$ be a closed and bounded complex $\epsilon$-disc about $r \in \R$ with $\epsilon > 0$, and fix $c_0 > 0$. Let $f,g \in \R^d[t]$ be such that:
    \begin{itemize}
        \item $f + c \cdot g$ has exactly two roots (counting multiplicity) in $B$ for all $c \in [0,c_0]$,
        \item $f$ has two real roots (counting multiplicity) in $B$,
        \item $f''(t) \cdot g(t) < 0$ for all $t \in B \cap \R$.
    \end{itemize}
    Then $f + g$ has two real and simple roots in $B$.
\end{lemma}
\begin{proof}
    By possibly negating both $f$ and $g$, we may assume that $f''(t) > 0$ for all $t \in B \cap \R = [a,b]$. Since $f$ has exactly two real roots $r_1 \leq r_2$ in $[a,b]$, this then implies $f(t) > 0$ for all $t \in [a,r_1) \cup (r_2,b]$ and $f(t) \leq 0$ for all $t \in [r_1,r_2]$. Further, $f''(t) \cdot g(t) < 0$ implies $g(t) < 0$ for $t \in [a,b]$, which in turn implies $(f + c \cdot g)(t) < 0$ for all $c \in (0,c_0]$ and all $t \in [r_1,r_2]$.

    Now note that since $B$ is closed and by continuity of roots, no roots of $f + c \cdot g$ can enter or leave $B$ as $c$ goes from $0$ to $c_0$. Thus $f + c \cdot g$ has sign pattern $(+,-,+)$ on $[a,b]$ for $c \in (0,c_0)$, and therefore $f + c_0 \cdot g$ has two real and simple roots in $B$.
\end{proof}

We now state and prove the main technical result of the paper. Note that for the case that the polynomials in $\bm{f}$ share no roots, Theorem~\ref{thm:simplex-main} is a direct corollary of this result.

\begin{theorem} \label{thm:main-compat-perturb}
    Let $\bm{f} = (f_1,\ldots,f_n) \subset \R^d[t]$ be proper and $3$-compatible, such that there is no $r \in \R \cup \{\infty\}$ for which $f(r) = 0$ for all $f \in \bm{f}$. Further, define $\bar{f} := \sum_{i=1}^n f_i$ and $\bm{g} = (f_1 + \epsilon \bar{f}, \ldots, f_n + \epsilon \bar{f})$. For all $\epsilon > 0$ small enough, $\bm{g}$ is compatible and every convex combination of $\bm{g}$ has only simple real roots.
\end{theorem}
\begin{proof}
    By possibly applying an $\SL_2(\R)$ perturbation to $\bm{f}$, we may assume that $f$ is of degree exactly $d$ for all $f \in \bm{f}$. Thus by Proposition~\ref{prop:3-compat-exists-root}, every convex combination of $\bm{f}$ is of degree either $d$ or $d-1$.
    
    So as to get a contradiction, suppose the desired result does not hold. Then by Proposition~\ref{prop:3-compat-real-rooted}, for every $\epsilon > 0$, either $\bm{g}$ is not $3$-compatible or else some convex combination of at most $2$ polynomials in $\bm{g}$ has a non-simple root. Thus for every $\epsilon = k^{-1}$ for $k \in \N$, let $H_k$ be a convex combination of at most $3$ polynomials\footnote{Using ``at most $3$'' instead of ``at most $2$'' here actually gives something stronger then needed for the simple roots condition of Proposition~\ref{prop:3-compat-real-rooted}.} in $\bm{g}$ which is either not real-rooted or which has a non-simple root. By permuting $\bm{f}$ and passing to a subsequence, we may assume without loss of generality that $H_k = a_k (f_1 + k^{-1} \bar{f}) + b_k (f_2 + k^{-1} \bar{f}) + c_k (f_3 + k^{-1} \bar{f})$, and that $a_k \to a$, $b_k \to b$, $c_k \to c$. By further defining $h_k = a_k f_1 + b_k f_2 + c_k f_3$, we have that $H_k = h_k + k^{-1} \bar{f}$. Finally we define $h = a f_1 + b f_2 + c f_3$, and $3$-compatibility of $\bm{f}$ implies $h$ is real-rooted and $h_k$ is real-rooted for all $k$.

    By Proposition~\ref{prop:double-root-signs} and the fact that no root is shared by every $f \in \bm{f}$, we have $h''(r) \cdot \bar{f}(r) < 0$ for every non-simple root $r$ of $h$. (Note that $h$ does not have a non-simple root at $\infty$ since $\deg(h) \in \{d-1,d\}$.) Let $r_1,\ldots,r_N \in \R$ be the distinct (finite) roots of $h$ with multiplicities $m_1,\ldots,m_N \in \{1,2\}$, let $I$ be the set of all $i$ such that $m_i = 2$, and let $U_1,\ldots,U_N$ be closed complex $\epsilon$-discs about $r_1,\ldots,r_N$ with $\epsilon > 0$ small enough so that:
    \begin{itemize}
        \item $U_1,\ldots,U_N$ are disjoint,
        \item $U_1,\ldots,U_N$ are contained in the open disc centered at $0$ of radius $\epsilon^{-1}$,
        \item for all $i \in I$ we have that $h''(t) \cdot \bar{f}(t) < 0$ for all $t \in U_i \cap \R$.
    \end{itemize}
    Further, let $S$ be the union of all $U_i$ for $i \in I$. By Theorem~\ref{thm:continuous-roots} applied to $h$ and Corollary~\ref{cor:continuous-stability} applied to $h''$, there exists $\delta > 0$ such that for $g \in \C^d[t]$:
    \begin{itemize}
        \item $\|g - h\| < \delta$ implies $g$ has $m_i$ roots (counted with multiplicity) in each $U_i$ and $\deg(g) - \deg(h) \leq 1$ roots with modulus greater than $\epsilon^{-1}$,
        \item $\|g'' - h''\| < \delta$ implies $g''$ has no roots in $S$.
    \end{itemize}
    Now, choose $k$ large enough such that $\|h_k'' - h''\| < \delta$ and $\|h_k + c \bar{f} - h\| < \delta$ for all $c \in [0,k^{-1}]$. By continuity, this implies for all $i \in I$ that $h_k''(t) \cdot \bar{f}(t) < 0$ for all $t \in U_i$. And for all $c \in [0,k^{-1}]$ and $i \in I$, the polynomial $h_k + c \bar{f}$ has exactly two roots in $U_i$. Applying Lemma~\ref{lem:two-real-roots} to $h_k,\bar{f}$ for all $i \in I$, we have that the roots of $H_k = h_k + k^{-1} \bar{f}$ are all real and simple. This is a contradiction, and thus proves the desired result.
\end{proof}

\begin{corollary} \label{cor:main-proper-case}
    Let $\bm{f} = (f_1,\ldots,f_n) \subset \R^d[t]$ be proper and $3$-compatible. Then $\bm{f}$ is compatible.
\end{corollary}
\begin{proof}
    By possibly dividing out simultaneous roots, we may assume without loss of generality that there is no $r \in \R \cup \{\infty\}$ for which $f(r) = 0$ for all $f \in \bm{f}$. Thus we may apply Theorem~\ref{thm:main-compat-perturb} for all $\epsilon > 0$ small enough, and by limiting $\epsilon \to 0$ we obtain the desired result.
\end{proof}

\subsection{The non-proper case}

Having proven the $(1) \iff (2)$ case of Theorem~\ref{thm:main} for proper $\bm{f}$ in Corollary~\ref{cor:main-proper-case}, we now handle non-proper $\bm{f}$.

\begin{theorem}
    Let $\bm{f} = (f_1,\ldots,f_n) \subset \R^d[t]$ be $3$-compatible. Then $\bm{f}$ is compatible.
\end{theorem}
\begin{proof}
    We prove this by induction on $n$, with the case of $n=3$ being immediate. If $\bm{f}$ is proper, then the desired result follows from Corollary~\ref{cor:main-proper-case}. Otherwise, there exists a convex combination of $\bm{f}$ which is the zero polynomial. That is, we have
    \[
        \tilde{f} = \sum_{k=1}^n c_k f_k = 0.
    \]
    Now let $g = \sum_{k=1}^n a_k f_k \not\equiv 0$ be any strictly positive convex combination of $\bm{f}$. Let $s > 0$ be maximal such that $a_k - s \cdot c_k \geq 0$ for all $k$. Thus $g = g - s \cdot \tilde{f}$ is (up to scalar) a convex combination of at most $n-1$ polynomials in $\bm{f}$. Thus $g$ is real-rooted by induction on $n$. Therefore every strictly positive convex combination of $\bm{f}$ is real-rooted, and by taking limits this implies $\bm{f}$ is compatible.
\end{proof}

\section{Proof of Theorem~\ref{thm:main}: $(2) \implies (3)$} \label{sec:main-2-implies-3}

We now give the remainder of the proof of Theorem~\ref{thm:main}. We first prove some basic root separation lemmas for proper families of $2$ and $3$ polynomials, and then we use these to reduce to the case of all positive leading coefficients. This allows us to utilize Theorem~\ref{thm:pos-leading-coeff-compat}.

Recall that if $f$ is a real-rooted polynomial then we denote its roots (with multiplicity) in $\R$ by $\lambda_1(f) \geq \lambda_2(f) \geq \cdots$.

\subsection{Basic lemmas}

The main result of this section in Lemma~\ref{lem:no-middle-roots-diff-sign} which says how the signs of leading coefficients affect the ordering of roots.

\begin{lemma} \label{lem:different-signs-distinct-roots}
    Let $\{f,g\} \subset \R^d[t]$ be proper and compatible such that $f,g$ are of degree exactly $d$ with leading coefficients of different signs. If all convex combinations of $f,g$ have only simple roots, then $\lambda_1(f) \neq \lambda_1(g)$.
\end{lemma}
\begin{proof}
    So as to get a contradiction, suppose $\lambda_1(f) = \lambda_1(g) = r$. Let $\tilde{f},\tilde{g} \in \R^{\tilde{d}}[t]$ be the polynomials obtained from $f,g$ respectively after dividing by $(t-r_i)$ for every shared root $r_i$ of $f,g$ (including $r$). Thus $\tilde{f}$ and $\tilde{g}$ have distinct largest roots such that without loss of generality $\lambda_1(\tilde{f}) < \lambda_1(\tilde{g}) < r$. Further, every convex combination of $\tilde{f},\tilde{g}$ is real-rooted, and since $\tilde{f},\tilde{g}$ have leading coefficients of different signs, no roots lie in the interval $I = (\max\{\lambda_1(\tilde{f}), \lambda_2(\tilde{g})\}, \lambda_1(\tilde{g}))$.

    Now consider $h_s = s \cdot \tilde{f} + (1-s) \cdot \tilde{g}$, and let $s_0 \in (0,1)$ be the unique value such that $h_{s_0}$ has a root at $\infty$. As $s$ goes from $0$ to $s_0$, the largest root of $h_s$ must go from $\lambda_1(\tilde{g})$ to $+\infty$, since no roots of $h_s$ lie in $I$ for all $s \in [0,1]$. Thus for some $s' \in (0,s_0)$, we have that $h_{s'}(r) = 0$. But this implies $s' \cdot f + (1-s') \cdot g$ has a double root at $r$, which is a contradiction.
\end{proof}

\begin{lemma} \label{lem:same-sign-distinct-roots}
    Let $\{f,g\} \subset \R^d[t]$ be proper and compatible such that $f,g$ are of degree exactly $d$ with leading coefficients of the same sign. If all non-negative linear combinations of $f,g$ have only simple roots, then $\lambda_2(f) < \lambda_1(g)$.
\end{lemma}
\begin{proof}
    If $\lambda_1(f) \leq \lambda_1(g)$ then the result immediately follows, and thus we may assume that $\lambda_1(g) < \lambda_1(f)$. So as to get a contradiction, suppose that $\lambda_1(g) \leq \lambda_2(f)$.
    
    If $\lambda_1(g) = \lambda_2(f)$, then choose $\phi \in \SL_2(\R)$ which rotates some point in the interval $(\lambda_2(f), \lambda_1(f))$ to $+\infty$. Thus $\phi \cdot f$ and $\phi \cdot g$ have degree $d$, leading coefficients of different signs, and $\lambda_1(\phi \cdot f) = \lambda_1(\phi \cdot g)$. This contradicts Lemma~\ref{lem:different-signs-distinct-roots}.

    Otherwise we have $\lambda_1(g) < \lambda_2(f)$. Let $h_s = s \cdot f + (1-s) \cdot g$ for all $s \in [0,1]$. Since $f,g$ have leading coefficients of the same sign, $h_s$ is of degree exactly $d$ and has no roots in the interval $(\lambda_1(g), \lambda_2(f))$ for all $s \in [0,1]$. Thus by continuity of roots, $h_s$ has no roots in $(\lambda_1(g), +\infty]$ as $s$ goes from $0$ to $1$. This contradicts the fact that $f$ has roots in this interval.
\end{proof}

\begin{lemma} \label{lem:no-middle-roots-diff-sign}
    Let $\{f,g,h\} \subset \R^d[t]$ be proper and compatible such that $f,g,h$ are of degree exactly $d$. Suppose further that $f,g$ have leading coefficients of the same sign, and $h$ has leading coefficient of a different sign. If all convex combinations of $f,g,h$ have only simple roots, then either $\lambda_1(f),\lambda_1(g) < \lambda_1(h)$ or $\lambda_1(h) < \lambda_1(f), \lambda_1(g)$.
\end{lemma}
\begin{proof}
    So as to get a contradiction, suppose without loss of generality that $\lambda_1(f) \leq \lambda_1(h) \leq \lambda_1(g)$. Let $p_s = s \cdot f + (1-s) \cdot g$ for $s \in [0,1]$. As $s$ goes from $0$ to $1$, the largest root of $p_s$ must go from $\lambda_1(g)$ to $\lambda_1(f)$, since $p_s$ cannot have any roots larger than $\lambda_1(g)$ for all $s \in [0,1]$. Thus for some $s' \in [0,1]$, we have that $p_{s'}(\lambda_1(h)) = 0$. However, this contradicts Lemma~\ref{lem:different-signs-distinct-roots} applied to $p_{s'}$ and $h$.
\end{proof}

\subsection{Reducing to the case of positive leading coefficients}

With the root separation lemmas of the previous section, we first reduce the $(2) \implies (3)$ case of Theorem~\ref{thm:main} to Theorem~\ref{thm:pos-leading-coeff-compat} for $3$-compatible families $\bm{f}$ for which all convex combinations have simple roots. We then prove the result in general using the perturbations described in Theorem~\ref{thm:main-compat-perturb}.

\begin{proposition} \label{prop:sl2-all-pos-leading-coeff}
    Let $\bm{f} = (f_1,\ldots,f_n) \subset \R^d[t]$ be proper and $3$-compatible, such that all convex combinations of any $3$ polynomials in $\bm{f}$ have only simple roots. Then there exists an $\SL_2(\R)$ transformation $\phi$ such that all polynomials in $\phi \cdot \bm{f}$ have degree exactly $d$ and leading coefficients of the same sign.
\end{proposition}
\begin{proof}
    By $\SL_2(\R)$ perturbation, we may assume that all polynomials in $\bm{f}$ are of degree exactly $d$. Let $r_+$ be the maximum $\lambda_1(f)$ over all $f \in \bm{f}$ which have positive leading coefficient, and let $r_-$ be the maximum $\lambda_1(f)$ over all $f \in \bm{f}$ which have negative leading coefficient. By possibly negating $\bm{f}$, without loss of generality we may assume that $r_+ < r_-$ (with strict inequality by Lemma~\ref{lem:different-signs-distinct-roots}).
    
    Let $f_+ \in \bm{f}$ have positive leading coefficient such that $\lambda_1(f_+) = r_+$, let $f_- \in \bm{f}$ have negative leading coefficient such that $\lambda_1(f_-) = r_-$, and let $f,g \in \bm{f}$ have negative leading coefficients. Then by applying Lemma~\ref{lem:no-middle-roots-diff-sign} to $f,f_-,f_+$, we have that $r_+ < r_-$ implies $r_+ < \lambda_1(f)$. And by applying Lemma~\ref{lem:same-sign-distinct-roots} to any $g,f$, we further have that $\lambda_2(g) < \lambda_1(f)$. Thus there exists $s \in \R$ which is strictly larger than $r_+$ and contained in $(\lambda_2(f), \lambda_1(f))$ for all $f \in \bm{f}$ which have negative leading coefficient. By choosing $\phi \in \SL_2(\R)$ which rotates $s$ to $+\infty$, we obtain the desired result.
\end{proof}

\begin{theorem} \label{thm:section-2-to-3}
    Let $\bm{f} = (f_1,\ldots,f_n) \subset \R^d[t]$ be proper and $3$-compatible. Then there is a real-rooted $g \in \R^d[t]$ such that $f_i \ll g$ for all $i$.
\end{theorem}
\begin{proof}
    We may assume without loss of generality that there is no $r \in \R \cup \{\infty\}$ for which $f(r) = 0$ for all $f \in \bm{f}$. Otherwise, we may divide each $f \in \bm{f}$ by $(t-r)$ for each simultaneous root $r$. After finding $\tilde{g}$ which satisfies the desired conclusion with these new polynomials, we simply multiply $\tilde{g}$ by these $(t-r)$ factors to obtain $g$.

    With this, we can apply Theorem~\ref{thm:main-compat-perturb} to $\bm{f}$. Letting $\bar{f} = \sum_{i=1}^n f_i$, this implies for all $\epsilon > 0$ small enough that $\bm{h} = (f_1 + \epsilon \bar{f}, \ldots, f_n + \epsilon \bar{f})$ is compatible and every convex combination of $\bm{h}$ has only simple roots. Thus by Proposition~\ref{prop:sl2-all-pos-leading-coeff}, there is some $\phi \in \SL_2(\R)$ (depending on $\epsilon$) such that all polynomials in $\phi \cdot \bm{h}$ have degree exactly $d$ and leading coefficients of the same sign. Therefore by Theorem~\ref{thm:pos-leading-coeff-compat} there exists $\tilde{g}_\epsilon$ such that $\phi \cdot (f_i + \epsilon \bar{f}) \ll \tilde{g}_\epsilon$ for all $i$, which implies $f_i + \epsilon \bar{f} \ll \phi^{-1} \cdot \tilde{g}_\epsilon =: g_\epsilon$ for all $i$.

    We now finally construct $g$. Note first that by scaling, we may assume that the maximum coefficient of $g_\epsilon$ in absolute value is exactly equal to $1$ for all $\epsilon > 0$ small enough. Since the set of polynomials $p \in \R^d[t]$ with coefficients in $[-1,1]$ is compact, we may choose a sequence $\epsilon_k \to 0$ such that $g_{\epsilon_k} \to g$ for some $g \in \R^d[t]$. Note that $g \not\equiv 0$ since $g_\epsilon$ always has a $\pm 1$ coefficient. Continuity of roots then implies $g$ satisfies the desired conclusion.
\end{proof}

\section{Proof of Theorem~\ref{thm:simplex-main}} \label{sec:simplex}

We now prove Theorem~\ref{thm:simplex-main} by showing that it is essentially a corollary of the $(1) \implies (3)$ case of Theorem~\ref{thm:main}. To do this, we need one extra lemma which shows that we can perturb interlacing polynomials so that they have simple roots. Recall that if $f$ is a real-rooted polynomial then we denote its roots (with multiplicity) in $\R$ by $\lambda_1(f) \geq \lambda_2(f) \geq \cdots$, and let $\|\cdot\|$ denote any norm on $\R^d[t]$ (since all are equivalent).

\begin{lemma} \label{lem:interlacing-perturb}
    Let $g, f_1,\ldots,f_n \in \R^d[t]$ be real-rooted such that $\bm{f} = (f_1,\ldots,f_n)$ is proper and that $f_i \ll g$ for all $i$. Then for all $\epsilon > 0$ small enough there exist real-rooted $\tilde{g}, \tilde{f}_1,\ldots,\tilde{f}_n \in \R^d[t]$ such that $\|g-\tilde{g}\| < \epsilon$, $\|f_i - \tilde{f}_i\| < \epsilon$ for all $i$, $(\tilde{f}_1,\ldots,\tilde{f}_n)$ is proper, $\tilde{f}_i \ll \tilde{g}$ for all $i$, and
    the roots (in $\R \cup \{\infty\}$) of $\tilde{g} \cdot \prod_{i=1}^n \tilde{f}_i$ are simple.
\end{lemma}
\begin{proof}
    Let $\phi \in \SO_2(\R) \subset \SL_2(\R)$ be such that $\phi \cdot g$ and $\phi \cdot f_i$ are all of degree exactly $d$ for all $i$. Since $\phi$ preserves a norm on $\R^d[t]$ by Remark~\ref{rem:rotations}, we may replace $g$ by $\phi \cdot g$ and $f_i$ by $\phi \cdot f_i$ for all $i$ without loss of generality. That is, we can act by $\phi$, determine $\tilde{g}$ and $\tilde{f}_i$ for all $i$, and then act by $\phi^{-1}$ to obtain the desired perturbations of the original polynomials. Further, by possibly negating all polynomials and permuting $\bm{f}$, we may assume without loss of generality that $g,f_1,\ldots,f_m$ have positive leading coefficients and $f_{m+1},\ldots,f_n$ have negative leading coefficients for some $0 \leq m \leq n$.
    
    Given any small $\epsilon > 0$, we now construct perturbations $\tilde{g},\tilde{f}_1,\ldots,\tilde{f}_n$ by perturbing roots via
    \[
        \lambda_j(\tilde{g}) = \lambda_j(g) - 3j \epsilon \quad \text{and} \quad 
        \begin{cases}
            \lambda_j(\tilde{f}_i) = \lambda_j(f_i) - (3j+i^{-1}) \epsilon, & 1 \leq i \leq m \\
            \lambda_j(\tilde{f}_k) = \lambda_j(f_k) - (3j-k^{-1}) \epsilon, & m+1 \leq k \leq n
        \end{cases}.
    \]
    Recall that all inequalities of the form $\lambda_{j+1}(g) \leq \lambda_j(f_i), \lambda_{j+1}(f_k) \leq \lambda_j(g)$ hold for valid $j$ and all $1 \leq i \leq m$ and $m+1 \leq k \leq n$. Thus we have
    \[
        \lambda_{j+1}(\tilde{g}) + 3\epsilon \leq \lambda_j(\tilde{f}_i) + i^{-1}\epsilon, \lambda_{j+1}(\tilde{f}_k) + (3-k^{-1})\epsilon \leq \lambda_j(\tilde{g}),
    \]
    which implies
    \[
        \lambda_{j+1}(\tilde{g}) < \lambda_j(\tilde{f}_i), \lambda_{j+1}(\tilde{f}_k) < \lambda_j(\tilde{g})
    \]
    for all $\epsilon > 0$ small enough. Note that this choice further implies that all roots of $\tilde{g} \cdot \prod_{i=1}^n \tilde{f}_i$ are simple, and that $\tilde{\bm{f}}$ is proper by compactness of $\Delta(\tilde{\bm{f}})$, for all $\epsilon > 0$ small enough. This yields the desired result.
\end{proof}

\begin{theorem} \label{thm:simplex-section}
    Let $\bm{f} = (f_1,\ldots,f_n) \subset \R^d[t]$ be proper and compatible. Then for all $\epsilon > 0$ small enough there exists $\tilde{\bm{f}} = (\tilde{f}_1,\ldots,\tilde{f}_n) \subset \R^d[t]$ such that $\|f_i - \tilde{f}_i\| < \epsilon$ for all $i$ and $\Delta(\tilde{\bm{f}})$ is contained in the interior of the set of real-rooted polynomials in $\R^d[t]$.
\end{theorem}
\begin{proof}
    By Theorem~\ref{thm:section-2-to-3} there is a real-rooted $g \in \R^d[t]$ such that $f_i \ll g$ for all $i$. By Lemma~\ref{lem:interlacing-perturb} we can perturb by $\epsilon > 0$ small enough to obtain $\tilde{g},\tilde{f}_1,\ldots,\tilde{f}_n$ so that $(\tilde{f}_1,\ldots,\tilde{f}_n)$ is proper, $\tilde{f}_i \ll \tilde{g}$ for all $i$, and the roots of $\tilde{g} \cdot \prod_{i=1}^n \tilde{f}_i$ are simple. By Corollary~\ref{cor:simple-3-compat} (and Theorem~\ref{thm:main}), every convex combination of $(\tilde{f}_1,\ldots,\tilde{f}_n)$ has only simple real roots. This implies the desired result.
\end{proof}

\section{Proof of Proposition~\ref{prop:main-compatible-triples}} \label{sec:proof-main-triples}

In this section we prove Proposition~\ref{prop:main-compatible-triples} via a simple observation relating compatibility of a non-proper triple $(f_1,f_2,f_3)$ to interlacing properties of the pair $(f_1,f_2)$.

\begin{lemma} \label{lem:3-to-2-nonproper}
    Let $f_1,f_2,f_3 \in \R^d[t]$ be real-rooted and such that $af_1 + bf_2 + cf_3 \equiv 0$ for some $a \geq 0$ and $b,c > 0$. Then $(f_1,f_2,f_3)$ is compatible if and only if either $f_1 \ll f_2$ or $f_2 \ll f_1$.
\end{lemma}
\begin{proof}
    For any $\alpha,\beta,\gamma \geq 0$ we have
    \[
        \alpha f_1 + (\beta-\gamma) f_2 = \left(\alpha + \frac{a\gamma}{b}\right) f_1 + \beta f_2 + \frac{c\gamma}{b} f_3
    \]
    and
    \[
        \alpha f_1 + \beta f_2 + \gamma f_3 = \left(\alpha - \frac{a\gamma}{c}\right) f_1 + \left(\beta - \frac{b\gamma}{c}\right) f_2.
    \]
    Thus $(f_1,f_2,f_3)$ is compatible if and only if every linear combination of $f_1,f_2$ is real-rooted. By Theorem~\ref{thm:hko}, this implies the desired result.
\end{proof}

\begin{proposition}
    Let $(f_1,f_2,f_3) \subset \R[t]$ be such that $f_i$ is real-rooted for all $i \in \{1,2,3\}$. Then $(f_1,f_2,f_3)$ is compatible if and only if one of the following conditions holds:
    \begin{enumerate}
        \item there is a real-rooted $g \in \R[t]$ such that $f_1,f_2,f_3 \ll g$,
        \item there is a choice of $\{i,j,k\} = \{1,2,3\}$, $a \geq 0$, and $b,c > 0$ such that $af_i + bf_j + cf_k \equiv 0$ and either $f_i \ll f_j$ or $f_j \ll f_i$.
    \end{enumerate}
\end{proposition}
\begin{proof}
    $(\implies)$. If $(f_1,f_2,f_3)$ is proper then condition $(1)$ follows from Theorem~\ref{thm:main}. If at least one of $f_1,f_2,f_3$ is the zero polynomial then $(1)$ also follows from the discussion following Theorem~\ref{thm:main}. Otherwise there is a choice of $\{i,j,k\} = \{1,2,3\}$ and $a \geq 0$, $b,c > 0$ such that $af_i + bf_j + cf_k = 0$, and thus $(2)$ follows from Lemma~\ref{lem:3-to-2-nonproper}.

    $(\impliedby)$. If $(1)$ holds then $(f_1,f_2,f_3)$ is compatible by Theorem~\ref{thm:main} and the discussion following Theorem~\ref{thm:main}. If $(2)$ holds, then $(f_1,f_2,f_3)$ is compatible by Lemma~\ref{lem:3-to-2-nonproper}.
\end{proof}

\section*{Acknowledgements} The first author acknowledges the support of the Natural Sciences and Engineering Research Council of Canada (NSERC), [funding reference number RGPIN-2023-03726]. Cette recherche a \'et\'e partiellement financ\'ee par le Conseil de recherches en sciences naturelles et en g\'enie du Canada (CRSNG), [num\'ero de r\'ef\'erence RGPIN-2023-03726].

\bibliographystyle{amsalpha}
\bibliography{bibliography}

\end{document}